\documentclass[12pt,a4paper]{amsart}
\usepackage{amsmath,amssymb,amsthm}
\usepackage{geometry}
\usepackage{hyperref}

\allowdisplaybreaks[4]
\DeclareMathOperator{\supp}{supp}
\DeclareMathOperator{\lm}{lm}

\DeclareMathOperator{\ara}{ara}
\DeclareMathOperator{\Ht}{ht}

\DeclareMathOperator{\rad}{rad}
\DeclareMathOperator{\Char}{char}

\newtheorem{theorem}{Theorem}[section]
\newtheorem{lemma}[theorem]{Lemma}
\newtheorem{proposition}[theorem]{Proposition}
\newtheorem{corollary}[theorem]{Corollary}

\theoremstyle{definition}
\newtheorem{definition}[theorem]{Definition}
\newtheorem{remark}[theorem]{Remark}
\newtheorem{example}[theorem]{Example}
\newtheorem{notation}[theorem]{Notation}

\title{Admissible bases and generic lattice ideals}
\author{Anargyros Katsabekis}
\address{Department of Mathematics, University of Ioannina, 45110 Ioannina, Greece}
\email{katsampekis@uoi.gr}

\begin{document}

\begin{abstract}
Let
$L\subseteq\mathbb Z^4$ be a rank-three lattice possessing an admissible basis, and let $S=K[x_1,\ldots,x_4]$ be the polynomial ring over a field $K$. We show that the associated lattice ideal $I_L \subset S$ is generic and that $S/I_L$ is a Cohen--Macaulay ring. We also identify the neighbors of
the origin in $L$ and the indispensable binomials of $I_L$. Moreover, we prove that the Betti numbers of
$S/I_L$ are $\beta_0(S/I_L)=1$, $\beta_1(S/I_L)=7$, $\beta_2(S/I_L)=12$, and $\beta_3(S/I_L)=6$. In addition,
we show that $I_L$ is a set-theoretic complete intersection over an
arbitrary field.

\end{abstract}

\maketitle

\section{Introduction}

Lattice ideals are binomial ideals appearing naturally in several areas of mathematics, including algebraic statistics, dynamical systems, and integer programming. Toric ideals form an important class of lattice ideals
and include the defining ideals of affine monomial curves. Let
$n_1,\ldots,n_r$ be positive integers with
$\gcd(n_1,\ldots,n_r)=1$. The defining ideal of the affine monomial curve
in the $r$-dimensional affine space $\mathbb{A}^r(K)$ parametrized by $x_i=t^{n_i}$ for $1\leq i\leq r$ is the toric ideal associated to the lattice
\[
\ker_{\mathbb{Z}}(n_1,\ldots,n_r)
=
\left\{
(u_1,\ldots,u_r)\in\mathbb{Z}^r
\;\middle|\;
u_1n_1+\cdots+u_rn_r=0
\right\}.
\]

Let $L\subseteq\mathbb{Z}^r$ be a positive lattice, and consider the lattice ideal
$I_L\subset K[x_1,\ldots,x_r]$, where $K$ is a field. In this paper, we focus on the case that $I_L$ is generic. For these ideals,
the minimal binomial generators of $I_L$ are in one-to-one
correspondence with the neighbors of the origin; see
Proposition~\ref{newindispensable}. 

Ojeda and Pis\'on-Casares studied generic toric ideals
associated with affine monomial curves $\mathbb{A}^4(K)$.
They proved that these ideals have at least seven minimal
binomial generators and determined the Betti numbers of
$K[x_1,\ldots,x_4]/I_L$ in terms of the minimal number
of generators \cite[Theorem~4.1]{OjPi}. In the same paper, they considered
specific bases of $\ker_{\mathbb{Z}}(n_1,\ldots,n_4)$
and applied \cite[Theorem~1]{BaSca} to show that the associated lattice ideals are generic.

The bases in the present work differ from those in \cite{OjPi}. We consider rank-three lattices $L\subseteq\mathbb Z^4$ admitting a basis $\{\mathbf{c}_1,\mathbf{c}_2,\mathbf{c}_3\}$ that satisfies the conditions \textup{(A1)}, \textup{(A2)}, and \textup{(A3)} of Definition~\ref{defadmissible}. We call such a basis admissible. Set
$\mathbf{u}_i=\mathbf{c}_i$ for $1\leq i\leq 3$, and define $\mathbf{u}_4=\mathbf{c}_1+\mathbf{c}_3$, $\mathbf{u}_5=\mathbf{c}_2+\mathbf{c}_3$, $\mathbf{u}_6=\mathbf{c}_1+\mathbf{c}_2+\mathbf{c}_3$, and $\mathbf{u}_7=\mathbf{c}_1+\mathbf{c}_2+2\mathbf{c}_3$.
With respect to a suitable graded reverse lexicographic order,
$\{B(\mathbf u_i)\mid 1\leq i\leq7\}$
is a Gr\"obner basis of $I_L$; in particular, $I_L$ is
generic (see Theorem~\ref{genericlattice}). These seven vectors also determine the neighbors of the origin, namely $\{\pm\mathbf{u}_i\mid 1\leq i\leq7\}$.
As a consequence, the binomials $B(\mathbf u_i)$, $1 \leq i \leq 7$, are indispensable and
form the unique minimal system of binomial generators of $I_L$
(see Theorem~\ref{computeneighbors}). Corollary~\ref{computebetti} determines the Betti
numbers of $K[x_1,\ldots,x_4]/I_L$.

Another goal of this paper is to use the results of Section~\ref{importantlattice} to show that $I_L$ is a set-theoretic complete intersection. In positive characteristic, every lattice ideal associated with a positive rank-three lattice in $\mathbb{Z}^4$ is a set-theoretic complete intersection; see \cite[Proposition~2.17]{LoVi}. In characteristic zero, this is not known in general, even for defining
ideals of monomial curves in $\mathbb{A}^4(K)$; see
\cite{HBr,Balanced,KaStci,AlmGo} for partial results. Theorem~\ref{computeara} shows that, for a lattice $L \subseteq \mathbb Z^4$ with an admissible basis, $I_L$ is a
set-theoretic complete intersection over an arbitrary field. Proposition~\ref{monomialcurve} determines when $I_L$ is the
defining ideal of an affine monomial curve.

Finally, for every integer $m \geq 1$ we construct a family
$L(m)\subseteq\mathbb{Z}^4$ of lattices with admissible bases. It is worth noting that $I_{L(m)}$ is a toric ideal exactly when $5$ does not divide $m$.

\section{Preliminaries}
\label{basics}

In this section, we fix the notation and recall results used in the paper.

Throughout this section, let $L\subseteq\mathbb{Z}^r$ be a
nonzero positive lattice, namely
$L\cap\mathbb{N}^r=\{\mathbf{0}\}$.
Given a vector $\mathbf{u}=(u_1,\ldots,u_r)\in\mathbb{Z}^r$,
we write $(\mathbf{u})_i=u_i$ for its $i$-th coordinate.
The \emph{support} of $\mathbf{u}$ is
\[
\supp(\mathbf{u})
=
\{i\in\{1,\ldots,r\}\mid(\mathbf{u})_i\neq0\}.
\]
The vector $\mathbf{u}$ is said to have \emph{full support} when
$\supp(\mathbf{u})=\{1,\ldots,r\}$.

Every vector $\mathbf{u}\in\mathbb{Z}^{r}$ can be written uniquely $\mathbf{u}=\mathbf{u}^{+}-\mathbf{u}^{-}$, where
$(\mathbf{u}^{+})_i=\max\{(\mathbf{u})_i,0\}$ and
$(\mathbf{u}^{-})_i=\max\{-(\mathbf{u})_i,0\}$ for each $i$. The vectors
$\mathbf{u}^{+}$ and $\mathbf{u}^{-}$ are called the
\emph{positive part} and the \emph{negative part} of $\mathbf{u}$, respectively.
They both lie in $\mathbb{N}^{r}$ and have disjoint supports.

For
$\mathbf{a}=(a_1,\ldots,a_r)\in\mathbb{N}^{r}$, we use the notation
$\mathbf{x}^{\mathbf{a}}=x_1^{a_1}\cdots x_r^{a_r}$. The \emph{lattice ideal} associated with $L$ is
\[
I_L=
\left\langle
\mathbf{x}^{\mathbf{u}^{+}}-\mathbf{x}^{\mathbf{u}^{-}}
\mid\mathbf{u}\in L
\right\rangle
\subseteq K[x_1,\ldots,x_r].
\]
The lattice ideal $I_L$ is called \emph{toric} if $\mathbb{Z}^{r}/L$ is
torsion free. We call $I_L$ \emph{generic} if it admits a generating set
consisting of binomials with full support. For $\mathbf{u}\in L$, write
$B(\mathbf{u})=
\mathbf{x}^{\mathbf{u}^{+}}-\mathbf{x}^{\mathbf{u}^{-}} \in I_L$.

By \cite[Corollary~2.2]{EisSt},
$\operatorname{ht}(I_L)=\operatorname{rank}(L)$.
Since $L$ is a positive lattice, there exist positive integers
$n_1,\ldots,n_r$ such that
$L \subseteq \ker_{\mathbb{Z}}(n_1,\ldots,n_r)$.
Thus, the ideal $I_L$ is homogeneous with respect to the grading
defined by $\deg(x_i)=n_i$, for every $1 \leq i \leq r$.

For
$\mathbf{a} \in\mathbb{Z}^{r}$ and $\mathbf{b}\in\mathbb{Z}^{r}$, we let
$\mathbf{a}\le\mathbf{b}$ whenever
$(\mathbf{a})_i\le(\mathbf{b})_i$ for all $1\le i\le r$. A nonzero vector $\mathbf u \in L$ is called a
\emph{neighbor of the origin} if every $\mathbf b \in L$ satisfying
$\mathbf{b}\le\mathbf{u}^{+}$ is equal to either $\mathbf0$ or
$\mathbf u$. We write $N(L)$ for the set of all neighbors of the origin.

A binomial $B\in I_L$ is said to be
\emph{indispensable} if every system of binomial generators of $I_L$
contains $B$ or $-B$. The following proposition makes the connection between neighbors of the origin and
indispensable binomials.

\begin{proposition}\label{newindispensable}
Let $\mathbf{u}\in L$ be a nonzero vector. Then the binomial
$B(\mathbf{u})$ is indispensable of $I_L$ if and only if
$\mathbf{u}$ is a neighbor of the origin.
\end{proposition}

\begin{proof}
Let $\{\mathbf d_1,\ldots,\mathbf d_s\}$ be a $\mathbb Z$-basis of $L$, and
write
$\mathbf u=a_1\mathbf d_1+\cdots+a_s\mathbf d_s$,
where $a_1,\ldots,a_s$ are integers. We shall denote by $D$ the $r \times s $matrix with columns
$\mathbf d_1,\ldots,\mathbf d_s$.

Every $\mathbf b\in L$ can be written uniquely in the form
$\mathbf b=c_1\mathbf d_1+\cdots+c_s\mathbf d_s$,
for some integers $c_1,\ldots,c_s$.

From the definition of the Hilbert basis $H(D)$ in \cite{Sab}, the vector $(a_1,\ldots,a_s)$ satisfies condition~\textup{(1)} of \cite[Theorem~2.2]{Sab} holds if and only if there is no
$(c_1,\ldots,c_s)\in\mathbb Z^s$, other than ${\bf 0}=(0,\ldots,0)$ and
$(a_1,\ldots,a_s)$, such that
\[
c_1\mathbf d_1+\cdots+c_s\mathbf d_s
\leq
\bigl(a_1\mathbf d_1+\cdots+a_s\mathbf d_s\bigr)^+.
\]
Notice that every $\mathbf b\in L$ can be written uniquely as 
$\mathbf b=c_1\mathbf d_1+\cdots+c_s\mathbf d_s$, for some integers $c_1,\ldots,c_s$. Thus, the last condition is equivalent to requiring that the only
elements $\mathbf b\in L$ satisfying
$\mathbf b\leq\mathbf u^+$ are $\mathbf0$ and $\mathbf u$.
Hence, $\mathbf u$ is a neighbor of the origin if and only if $(a_1,\ldots,a_s)$ satisfies condition~\textup{(1)}.

On the other hand, condition~\textup{(5)} of \cite[Theorem~2.2]{Sab} states that
$B(\mathbf u)$
is indispensable of $I_L$. By \cite[Theorem~2.2]{Sab}, conditions~\textup{(1)}
and~\textup{(5)} are equivalent, so $\mathbf{u}$ is a neighbor of the origin if and only if $B(\mathbf u)$ is indispensable.
\end{proof}

The \emph{arithmetical rank} of $I_L$, denoted by $\ara(I_L)$, is the smallest integer $s$ for which there exist
$f_1,\ldots,f_s\in I_L$ such that
$\rad(I_L)=\rad(f_1,\ldots,f_s)$.
By Krull's principal ideal theorem, the inequality
$\Ht(I_L)\le\ara(I_L)$ holds. We call $I_L$ a
\emph{set-theoretic complete intersection} when
$\ara(I_L)=\Ht(I_L)$.

For a positive lattice $L\subseteq\mathbb{Z}^{r}$ and an index
$1\le i\le r$ for which there exists $\mathbf{v}\in L$ with
$\supp(\mathbf{v}^{+})=\{i\}$, define
\[
\gamma_i(L)
=
\min\left\{
(\mathbf{v})_i
\mid
\mathbf{v}\in L,\,
\supp(\mathbf{v}^{+})=\{i\}
\right\}.
\]
Notice that the condition
$\supp(\mathbf{v}^{+})=\{i\}$ is equivalent to
$(\mathbf{v})_i>0$ and $(\mathbf{v})_j\le0$ for every $j\neq i$.

The following criterion gives a sufficient condition under which $I_L$ is a set-theoretic complete intersection when the characteristic of $K$ is zero.

\begin{theorem}[{\cite[Theorem~1.5]{KaStci}}]
\label{basicstci}
Assume that $\Char(K)=0$. Let $L\subseteq\mathbb{Z}^4$ be a positive
lattice of rank three. For each $1\le j\le4$, let
$\mathbf{v}_j\in L$ satisfy $\supp(\mathbf{v}_j^{+})=\{j\}$. Suppose
that the lattice $L$ is generated by the vectors $\mathbf{v}_1,\ldots,\mathbf{v}_4$.
Assume further that there are positive integers $d_1,\ldots,d_4$
satisfying
$\sum_{j=1}^{4}d_j\mathbf{v}_j=\mathbf{0}$,
and that
$|(\mathbf{v}_j)_i|<\gamma_i(L)$
whenever $1\le i,j\le4$ and $i\neq j$. If two distinct indices
$j_1$ and $j_2$ satisfy $d_{j_1}=d_{j_2}=1$, then $I_L$ is a
set-theoretic complete intersection.
\end{theorem}

\section{Lattices with admissible bases}
\label{importantlattice}

We start by defining the central notion of this section, namely an admissible basis of a lattice.

\begin{definition}\label{defadmissible}
Let $L\subseteq\mathbb{Z}^4$ be a rank-three lattice. A
$\mathbb{Z}$-basis $\{\mathbf{c}_1,\mathbf{c}_2,\mathbf{c}_3\}$ of $L$
is called \emph{admissible} if $\mathbf{c}_1=(\alpha_{11},-\alpha_{12},-\alpha_{13},-\alpha_{14})$, $\mathbf{c}_2=(-\alpha_{21},-\alpha_{22},-\alpha_{23},\alpha_{24})$, and $\mathbf{c}_3=(-\alpha_{31},\alpha_{32},-\alpha_{33},-\alpha_{34})$,
where all $\alpha_{ij}$ are positive integers, and
\begin{enumerate}
\item[\textup{(A1)}]
$\alpha_{11}>\alpha_{21}+2\alpha_{31}$;

\item[\textup{(A2)}]
$\max\{\alpha_{12},\alpha_{22}\}<\alpha_{32}
<\alpha_{12}+\alpha_{22}$;

\item[\textup{(A3)}]
$\alpha_{24}>\alpha_{14}+2\alpha_{34}$.
\end{enumerate}
\end{definition}

\begin{example}\label{explanatoryexample}
Let $L$ be the lattice generated by the vectors
$\mathbf{c}_1=(4,-2,-1,-1)$, $\mathbf{c}_2=(-1,-2,-1,4)$, and $\mathbf{c}_3=(-1,3,-1,-1)$.
The vectors $\mathbf{c}_i$, $1 \leq i \leq 3$, are linearly independent and hence
form a $\mathbb{Z}$-basis of $L$. Here
$\alpha_{11}=4>3=\alpha_{21}+2\alpha_{31}$,
$\max\{\alpha_{12},\alpha_{22}\}=2<3=\alpha_{32}<4
=\alpha_{12}+\alpha_{22}$, and
$\alpha_{24}=4>3=\alpha_{14}+2\alpha_{34}$. So $\{\mathbf{c}_1, \mathbf{c}_2, \mathbf{c}_3\}$ is an admissible basis of $L$.
\end{example}

\begin{proposition}
\label{positivelattice}
If $\{\mathbf{c}_1,\mathbf{c}_2,\mathbf{c}_3\}$ is an admissible basis
of the lattice $L\subseteq\mathbb{Z}^4$, then $L$ is positive.
\end{proposition}

\begin{proof}
Let
$\mathbf b=p\mathbf c_1+q\mathbf c_2+r\mathbf c_3
\in L\cap\mathbb N^4$, where $p$, $q$, and $r$ are integers.
We will show that $p=q=r=0$.

Our first goal is to show that $p \geq 0$ and $q \geq 0$. Suppose that at least one of $p$ and $q$ is negative. Then either $p<0$ and $p\le q$ or $q<0$ and $q \leq p$.

If $p<0$ and $p\le q$, then 
\[
0\le p\alpha_{11}-q\alpha_{21}-r\alpha_{31}
\le p(\alpha_{11}-\alpha_{21})-r\alpha_{31}.
\]
Since $\alpha_{11}-\alpha_{21}>2\alpha_{31}$, we have $p(\alpha_{11}-\alpha_{21})<2p\alpha_{31}$, and therefore $(2p-r)\alpha_{31}>0$. Thus $r<2p$.
On the other hand, 
\[
0\le -p\alpha_{12}-q\alpha_{22}+r\alpha_{32}
\le -p(\alpha_{12}+\alpha_{22})+r\alpha_{32}.
\]
Hence
\[
r\ge p\frac{\alpha_{12}+\alpha_{22}}{\alpha_{32}}.
\]
Furthermore,
$\alpha_{12}+\alpha_{22}<2\alpha_{32}$, so $p\frac{\alpha_{12}+\alpha_{22}}{\alpha_{32}}>2p$, and therefore $r>2p$, a contradiction.

If $q<0$ and $q\le p$, then
\[
0\le -p\alpha_{14}+q\alpha_{24}-r\alpha_{34}
\le q(\alpha_{24}-\alpha_{14})-r\alpha_{34}.
\]
Since $\alpha_{24}-\alpha_{14}>2\alpha_{34}$, we have $r<2q$. Also
\[
0\le -p\alpha_{12}-q\alpha_{22}+r\alpha_{32}
\le -q(\alpha_{12}+\alpha_{22})+r\alpha_{32}.
\]
Thus
\[
r\ge q\frac{\alpha_{12}+\alpha_{22}}{\alpha_{32}},
\] and therefore $r>2q$, a contradiction.

Consequently, $p \geq 0$ and $q \geq 0$. We have 
$r\alpha_{32}\ge p\alpha_{12}+q\alpha_{22}$,
so $r\ge0$. Moreover, $-p\alpha_{13}-q\alpha_{23}-r\alpha_{33}\geq 0$. Since $\alpha_{13}$, $\alpha_{23}$, and $\alpha_{33}$ are positive integers, we deduce that $p=q=r=0$.
Hence,
$L\cap\mathbb N^4=\{\mathbf0\}$.
\end{proof}

Throughout the remainder of this section, let $\{\mathbf{c}_1, \mathbf{c}_2, \mathbf{c}_3\}$ be an admissible basis of the lattice $L \subseteq \mathbb{Z}^4$. Set
$\mathbf{u}_i=\mathbf{c}_i$ for $1 \leq i \leq 3$ and consider four more vectors in $L$, namely
$\mathbf{u}_4=\mathbf{c}_1+\mathbf{c}_3$, $\mathbf{u}_5=\mathbf{c}_2+\mathbf{c}_3$, $\mathbf{u}_6=\mathbf{c}_1+\mathbf{c}_2+\mathbf{c}_3$, and $\mathbf{u}_7=\mathbf{c}_1+\mathbf{c}_2+2\mathbf{c}_3$.

\begin{remark}\label{fullvectors} Using conditions \textup{(A1)}, \textup{(A2)}, and \textup{(A3)}, we deduce that each of the vectors $\mathbf{u}_1,\ldots,\mathbf{u}_7$
has full support. In addition, $(\mathbf{u}_i^+)_3=0$ and
$(\mathbf{u}_i^-)_3>0$ for every $1\le i\le7$ since $(\mathbf{u}_i)_3<0$.
\end{remark}

Since $L$ is positive, there exist positive integers $n_i$, $1 \leq i \leq 4$, such that $I_L$ is homogeneous with respect to the grading $\deg(x_i)=n_i$. 

Consider the graded reverse lexicographic order $\prec$ on $K[x_1,\ldots,x_4]$ with $\deg(x_i)=n_i$ and $x_1\succ x_2\succ x_4\succ x_3$. More precisely, for vectors ${\bf a}=(a_1,\ldots,a_4)$ and ${\bf b}=(b_1,\ldots,b_4)$ in $\mathbb N^4$, we have
$\mathbf{x}^{\mathbf{a}}\prec\mathbf{x}^{\mathbf{b}}$ if either
$\sum_{i=1}^4 a_i n_i<\sum_{i=1}^4 b_i n_i$, or $\sum_{i=1}^4 a_i n_i=\sum_{i=1}^4 b_i n_i$ and the last nonzero entry of
$(a_1-b_1,a_2-b_2,a_4-b_4,a_3-b_3)$ is positive.

For $1\leq i\leq7$, set $g_i=B(\mathbf u_i) \in I_L$, namely 
\[
g_1=x_1^{\alpha_{11}}-x_2^{\alpha_{12}}x_3^{\alpha_{13}}x_4^{\alpha_{14}},
\
g_2=x_4^{\alpha_{24}}-x_1^{\alpha_{21}}x_2^{\alpha_{22}}x_3^{\alpha_{23}},
\
g_3=x_2^{\alpha_{32}}-x_1^{\alpha_{31}}x_3^{\alpha_{33}}x_4^{\alpha_{34}},
\]
\[
g_4=x_1^{\alpha_{11}-\alpha_{31}}x_2^{\alpha_{32}-\alpha_{12}}
-x_3^{\alpha_{13}+\alpha_{33}}x_4^{\alpha_{14}+\alpha_{34}},\
g_5=x_2^{\alpha_{32}-\alpha_{22}}x_4^{\alpha_{24}-\alpha_{34}}
-x_1^{\alpha_{21}+\alpha_{31}}x_3^{\alpha_{23}+\alpha_{33}},
\]
\[
g_6=x_1^{\alpha_{11}-\alpha_{21}-\alpha_{31}}
x_4^{\alpha_{24}-\alpha_{14}-\alpha_{34}}
-x_2^{\alpha_{12}+\alpha_{22}-\alpha_{32}}
x_3^{\alpha_{13}+\alpha_{23}+\alpha_{33}},\ \textrm{and}
\]
\[
g_7=x_1^{\alpha_{11}-\alpha_{21}-2\alpha_{31}}
x_2^{2\alpha_{32}-\alpha_{12}-\alpha_{22}}
x_4^{\alpha_{24}-\alpha_{14}-2\alpha_{34}}
-x_3^{\alpha_{13}+\alpha_{23}+2\alpha_{33}}.
\]
Since $u_i \in L$, each $g_i$ belongs to $I_L$. Moreover, by Remark~\ref{fullvectors} $x_3$ divides $\mathbf{x}^{\mathbf{u}_i^-}$ and $x_3$ does not divide $\mathbf{x}^{\mathbf{u}_i^+}$ for every $1 \leq i \leq 7$, so $\lm(g_i)=\mathbf{x}^{\mathbf{u}_i^+}$.

\begin{theorem}\label{genericlattice}
Let $\{\mathbf{c}_1,\mathbf{c}_2,\mathbf{c}_3\}$ be an admissible basis
of a lattice $L\subseteq\mathbb{Z}^4$. Then
$G=\{g_1,\ldots,g_7\}$ is a Gr\"obner basis of $I_L$ with
respect to $\prec$, and $I_L$ is generic.
\end{theorem}

\begin{proof}
Let $J=\langle g_1,\ldots,g_7\rangle\subseteq I_L$. Our first goal is to show that $G$ is a Gr\"obner basis for $J$ with respect to $\prec$. The leading monomials of the elements of $G$ are
$$\lm(g_1)=x_1^{\alpha_{11}},\ \lm(g_2)=x_4^{\alpha_{24}},\ \lm(g_3)=x_2^{\alpha_{32}},\ \lm(g_4)=x_1^{\alpha_{11}-\alpha_{31}}x_2^{\alpha_{32}-\alpha_{12}},$$ $$\lm(g_5)=x_2^{\alpha_{32}-\alpha_{22}}x_4^{\alpha_{24}-\alpha_{34}},\ \lm(g_6)=x_1^{\alpha_{11}-\alpha_{21}-\alpha_{31}}x_4^{\alpha_{24}-\alpha_{14}-\alpha_{34}},$$ $$\lm(g_7)=x_1^{\alpha_{11}-\alpha_{21}-2\alpha_{31}}x_2^{2\alpha_{32}-\alpha_{12}-\alpha_{22}}x_4^{\alpha_{24}-\alpha_{14}-2\alpha_{34}}.$$

Since the leading monomials of $g_i$ and $g_j$ are relatively
prime for each of the pairs
\[
\{g_1,g_2\},\{g_1,g_3\},\{g_1,g_5\},
\{g_2,g_3\},\{g_2,g_4\},\{g_3,g_6\},
\]
their $S$-polynomials reduce to zero by \cite[Chapter~2, \S9, Proposition~4]{CoLiOs}.

Moreover,
\[
S(g_1,g_4)=-x_3^{\alpha_{13}}x_4^{\alpha_{14}}g_3 \stackrel{g_3}{\longrightarrow} 0,
\
S(g_1,g_6)=-x_2^{\alpha_{12}+\alpha_{22}-\alpha_{32}}
x_3^{\alpha_{13}}g_5 \stackrel{g_5}{\longrightarrow} 0,
\]
\[
S(g_2,g_5)=-x_1^{\alpha_{21}}x_3^{\alpha_{23}}g_3 \stackrel{g_3}{\longrightarrow} 0,
\
S(g_2,g_6)=-x_2^{\alpha_{12}+\alpha_{22}-\alpha_{32}}
x_3^{\alpha_{23}}g_4 \stackrel{g_4}{\longrightarrow} 0,
\]
\[
S(g_3,g_4)=-x_3^{\alpha_{33}}x_4^{\alpha_{34}}g_1 \stackrel{g_1}{\longrightarrow} 0,
\
S(g_3,g_5)=-x_1^{\alpha_{31}}x_3^{\alpha_{33}}g_2 \stackrel{g_2}{\longrightarrow} 0,
\]
\[
S(g_3,g_7)=-x_3^{\alpha_{33}}g_6 \stackrel{g_6}{\longrightarrow} 0,
\
S(g_4,g_6)=-x_3^{\alpha_{13}+\alpha_{33}}g_2 \stackrel{g_2}{\longrightarrow} 0,
\]
\[
S(g_4,g_7)=-x_3^{\alpha_{13}+\alpha_{33}}g_5 \stackrel{g_5}{\longrightarrow} 0,
\
S(g_5,g_6)=-x_3^{\alpha_{23}+\alpha_{33}}g_1 \stackrel{g_1}{\longrightarrow} 0,
\]
\[
S(g_5,g_7)=-x_3^{\alpha_{23}+\alpha_{33}}g_4 \stackrel{g_4}{\longrightarrow} 0,
\
S(g_6,g_7)=-x_3^{\alpha_{13}+\alpha_{23}+\alpha_{33}}g_3 \stackrel{g_3}{\longrightarrow} 0.
\]

Consider the pair $\{g_1,g_7\}$. The least common multiple of $\lm(g_1)$ and $\lm(g_7)$ is
$$x_1^{\alpha_{11}}x_2^{2\alpha_{32}-\alpha_{12}-\alpha_{22}}
x_4^{\alpha_{24}-\alpha_{14}-2\alpha_{34}},$$ and hence
\[
S(g_1,g_7)
=x_1^{\alpha_{21}+2\alpha_{31}}
x_3^{\alpha_{13}+\alpha_{23}+2\alpha_{33}}
-x_2^{2\alpha_{32}-\alpha_{22}}
x_3^{\alpha_{13}}x_4^{\alpha_{24}-2\alpha_{34}}.
\]
Here $\lm(S(g_1,g_7))=x_2^{2\alpha_{32}-\alpha_{22}}
x_3^{\alpha_{13}}x_4^{\alpha_{24}-2\alpha_{34}}$. Since
$2\alpha_{32}-\alpha_{22}>\alpha_{32}$ by condition~\textup{(A2)},
$\lm(g_3)$ divides $\lm(S(g_1,g_7))$ and also
\[
S(g_1,g_7)\stackrel{g_3}{\longrightarrow}
-x_1^{\alpha_{31}}x_3^{\alpha_{13}+\alpha_{33}}g_5
\stackrel{g_5}{\longrightarrow}0.
\]

Consider the pair $\{g_2,g_7\}$. The least common multiple of
$\lm(g_2)$ and $\lm(g_7)$ is
$$x_1^{\alpha_{11}-\alpha_{21}-2\alpha_{31}}
x_2^{2\alpha_{32}-\alpha_{12}-\alpha_{22}}
x_4^{\alpha_{24}},$$ and therefore
\[
S(g_2,g_7)
=x_3^{\alpha_{13}+\alpha_{23}+2\alpha_{33}}
x_4^{\alpha_{14}+2\alpha_{34}}
-x_1^{\alpha_{11}-2\alpha_{31}}
x_2^{2\alpha_{32}-\alpha_{12}}x_3^{\alpha_{23}}.
\]
Here $\lm(S(g_2,g_7))=x_1^{\alpha_{11}-2\alpha_{31}}
x_2^{2\alpha_{32}-\alpha_{12}}x_3^{\alpha_{23}}$. Then
$\lm(g_3)$ divides $\lm(S(g_2,g_7))$ and also
\[
S(g_2,g_7)\stackrel{g_3}{\longrightarrow}
-x_3^{\alpha_{23}+\alpha_{33}}x_4^{\alpha_{34}}g_4
\stackrel{g_4}{\longrightarrow}0.
\]

It remains to consider the pair $\{g_4,g_5\}$. The least common multiple of
$\lm(g_4)$ and $\lm(g_5)$ is
\[
\operatorname{lcm}(\lm(g_4),\lm(g_5))
=x_1^{\alpha_{11}-\alpha_{31}}
x_2^{\alpha_{32}-\min\{\alpha_{12},\alpha_{22}\}}
x_4^{\alpha_{24}-\alpha_{34}}.
\] Then $S(g_4,g_5)
=
x_2^{\alpha_{12}-\min\{\alpha_{12},\alpha_{22}\}}
x_4^{\alpha_{24}-\alpha_{34}}g_4-
x_1^{\alpha_{11}-\alpha_{31}}
x_2^{\alpha_{22}-\min\{\alpha_{12},\alpha_{22}\}}g_5$, and therefore $S(g_4,g_5)
=
x_1^{\alpha_{11}+\alpha_{21}}
x_2^{\alpha_{22}-\min\{\alpha_{12},\alpha_{22}\}}
x_3^{\alpha_{23}+\alpha_{33}}
-
x_2^{\alpha_{12}-\min\{\alpha_{12},\alpha_{22}\}}
x_3^{\alpha_{13}+\alpha_{33}}
x_4^{\alpha_{14}+\alpha_{24}}$. Using $\alpha_{22}-\min\{\alpha_{12},\alpha_{22}\}
=\max\{0,\alpha_{22}-\alpha_{12}\}$ and $\alpha_{12}-\min\{\alpha_{12},\alpha_{22}\}
=\max\{0,\alpha_{12}-\alpha_{22}\}$ we get
$S(g_4,g_5)=x_1^{\alpha_{11}+\alpha_{21}}
x_2^{\max\{0,\alpha_{22}-\alpha_{12}\}}
x_3^{\alpha_{23}+\alpha_{33}}
-
x_2^{\max\{0,\alpha_{12}-\alpha_{22}\}}
x_3^{\alpha_{13}+\alpha_{33}}
x_4^{\alpha_{14}+\alpha_{24}}$.

Suppose that $\alpha_{23} \leq \alpha_{13}$, so $\lm(S(g_4,g_5))=x_1^{\alpha_{11}+\alpha_{21}}
x_2^{\max\{0,\alpha_{22}-\alpha_{12}\}}
x_3^{\alpha_{23}+\alpha_{33}}$. Using $
\max\{0,\alpha_{22}-\alpha_{12}\}+\alpha_{12}
=
\max\{0,\alpha_{12}-\alpha_{22}\}+\alpha_{22}$, we obtain  \[
 \begin{aligned}
S(g_4,g_5)&\stackrel{g_1}{\longrightarrow}
x_1^{\alpha_{21}}
x_2^{\max\{0,\alpha_{22}-\alpha_{12}\}+\alpha_{12}}
x_3^{\alpha_{13}+\alpha_{23}+\alpha_{33}}
x_4^{\alpha_{14}}-
x_2^{\max\{0,\alpha_{12}-\alpha_{22}\}}
x_3^{\alpha_{13}+\alpha_{33}}
x_4^{\alpha_{14}+\alpha_{24}}\\
&=-
x_2^{\max\{0,\alpha_{12}-\alpha_{22}\}}
x_3^{\alpha_{13}+\alpha_{33}}
x_4^{\alpha_{14}}g_2
\stackrel{g_2}{\longrightarrow}0.
\end{aligned}
\]

Suppose that $\alpha_{13}<\alpha_{23}$, so $\lm(S(g_4,g_5))=x_2^{\max\{0,\alpha_{12}-\alpha_{22}\}}
x_3^{\alpha_{13}+\alpha_{33}}
x_4^{\alpha_{14}+\alpha_{24}}$. Using $\max\{0,\alpha_{22}-\alpha_{12}\}+\alpha_{12}
=
\max\{0,\alpha_{12}-\alpha_{22}\}+\alpha_{22}$, we obtain \[
\begin{aligned}
S(g_4,g_5)&\stackrel{g_2}{\longrightarrow}
x_1^{\alpha_{11}+\alpha_{21}}
x_2^{\max\{0,\alpha_{22}-\alpha_{12}\}}
x_3^{\alpha_{23}+\alpha_{33}}-
x_1^{\alpha_{21}}
x_2^{\max\{0,\alpha_{12}-\alpha_{22}\}+\alpha_{22}}
x_3^{\alpha_{13}+\alpha_{23}+\alpha_{33}}
x_4^{\alpha_{14}}\\
&=
x_1^{\alpha_{21}}
x_2^{\max\{0,\alpha_{22}-\alpha_{12}\}}
x_3^{\alpha_{23}+\alpha_{33}}g_1
\stackrel{g_1}{\longrightarrow}0.
\end{aligned}
\]
By \cite[Chapter~2, \S6, Theorem~6]{CoLiOs}, $G$ is a Gr\"obner basis of $J$.

Our second goal is to show that $J=I_L$. Since none of
the leading monomials of the elements of $G$ is divisible by $x_3$, we have $
J:x_3^\infty=J$ by \cite[Lemma~12.1]{BSt}. Moreover, \cite[Corollary~2.5]{EisSt} implies that $I_L
=
\langle g_1,g_2,g_3\rangle:(x_1x_2x_3x_4)^\infty$ because $L$ is generated by the vectors $\mathbf{c}_1$, $\mathbf{c}_2$, and $\mathbf{c}_3$. Thus, $I_L \subseteq J:(x_1x_2x_3x_4)^\infty$. It is enough to show that $I_L \subseteq J$. Let $f\in I_L$. Then there exists an integer $N\ge1$ such that
$(x_1x_2x_3x_4)^Nf\in J$. Also the binomial $x_3^{N(\alpha_{13}+\alpha_{23}+2\alpha_{33})}
-
x_1^{N(\alpha_{11}-\alpha_{21}-2\alpha_{31})}
x_2^{N(2\alpha_{32}-\alpha_{12}-\alpha_{22})}
x_4^{N(\alpha_{24}-\alpha_{14}-2\alpha_{34})}$ belongs to $J$. Furthermore, the polynomial
$x_3^{N(\alpha_{13}+\alpha_{23}+2\alpha_{33})+N}f
-
x_1^{N(\alpha_{11}-\alpha_{21}-2\alpha_{31})}
x_2^{N(2\alpha_{32}-\alpha_{12}-\alpha_{22})} x_3^N
x_4^{N(\alpha_{24}-\alpha_{14}-2\alpha_{34})}f$ belongs to $J$. From Conditions \textup{(A1)}, \textup{(A2)} and \textup{(A3)}, we have $\alpha_{11}-\alpha_{21}-2\alpha_{31}>0$, $2\alpha_{32}-\alpha_{12}-\alpha_{22}>0$ and $\alpha_{24}-\alpha_{14}-2\alpha_{34}>0$. Note that \begin{align*}
&x_1^{N(\alpha_{11}-\alpha_{21}-2\alpha_{31})}
 x_2^{N(2\alpha_{32}-\alpha_{12}-\alpha_{22})}
 x_3^N
 x_4^{N(\alpha_{24}-\alpha_{14}-2\alpha_{34})}f\\
&=
\bigl(
x_1^{\alpha_{11}-\alpha_{21}-2\alpha_{31}-1}
x_2^{2\alpha_{32}-\alpha_{12}-\alpha_{22}-1}
x_4^{\alpha_{24}-\alpha_{14}-2\alpha_{34}-1}
\bigr)^N \cdot (x_1x_2x_3x_4)^Nf
\in J,
\end{align*} since $(x_1x_2x_3x_4)^Nf \in J$. Therefore $x_3^{N(\alpha_{13}+\alpha_{23}+2\alpha_{33})+N}f\in J$. Consequently, $f \in J:x_3^\infty$, and hence $f \in J$ because
$J:x_3^\infty=J$.

Since $G=\{g_1,\ldots,g_7\}$ is a Gr\"obner basis for $I_L$ with respect to $\prec$ consisting entirely of binomials with full support, the ideal $I_L$ is generic.
\end{proof}

\begin{theorem}\label{computeneighbors}
Let $\{\mathbf{c}_1,\mathbf{c}_2,\mathbf{c}_3\}$ be an admissible basis
of a lattice $L\subseteq\mathbb{Z}^4$. Then
\[
N(L)=\{\pm\mathbf{u}_i\mid1\le i\le7\},
\]
and $B(\mathbf{u}_i)$,
$1\le i\le7$, are the indispensable binomials of $I_L$. 
\end{theorem}

\begin{proof}
By Theorem~\ref{genericlattice}, $G=\{g_1,\ldots,g_7\}$ is a generating set for $I_L$. A generic lattice ideal of height $3$ has at least seven minimal binomial generators; see \cite[Theorem~4.1]{OjPi} and the proof of
\cite[Corollary~1.3]{OjExamples}. Since $G$ already has seven elements, it is a minimal binomial generating set of $I_L$. By \cite[Remark~4.4(3)]{PeevaSt}, $g_1,\ldots,g_7$ are precisely the indispensable
binomials of $I_L$. From Proposition~\ref{newindispensable} we have that
\[
N(L)=\{\pm\mathbf{u}_i\mid1\le i\le7\}. \qedhere
\]

\end{proof}

\begin{corollary}\label{computebetti}
Let $S=K[x_1,\ldots,x_4]$ and let $\{\mathbf{c}_1,\mathbf{c}_2,\mathbf{c}_3\}$ be an admissible basis
of a lattice $L\subseteq\mathbb{Z}^4$. Then the Betti numbers of $S/I_L$ are
$\beta_0(S/I_L)=1$, $\beta_1(S/I_L)=7$, $\beta_2(S/I_L)=12$, $\beta_3(S/I_L)=6$, and $\beta_i(S/I_L)=0$ for every $i\ge4$.
\end{corollary}

\begin{proof}
For a generic lattice ideal $J$ of height three in $S$ with $t$ minimal binomial generators, the
Betti numbers of $J$ are $\beta_0(J)=t$, $\beta_1(J)=2t-2$, $\beta_2(J)=t-1$, and $\beta_i(J)=0$ for every $i\ge3$; see \cite[Theorem~4.1]{OjPi} and the proof of
\cite[Corollary~1.3]{OjExamples}. Moreover, $\beta_0(S/J)=1$, $\beta_1(S/J)=t$, $\beta_2(S/J)=2t-2$, $\beta_3(S/J)=t-1$ and $\beta_i(S/J)=0$ for every $i\ge4$. The minimal number of generators of $I_L$ is $t=7$ by Theorem~\ref{computeneighbors}, and therefore
$\beta_1(S/I_L)=7$, $\beta_2(S/I_L)=12$, $\beta_3(S/I_L)=6$, and $\beta_i(S/I_L)=0$ for every $i\ge4$.
\end{proof}

\begin{remark}\label{Cohen-Macaulay}
The rank of $L$ is three, so $\Ht(I_L)=3$, and therefore $S/I_L$ is a one-dimensional ring. By Corollary~\ref{computebetti}, $\beta_3(S/I_L)=6\ne0$ and $\beta_i(S/I_L)=0$ for $i\ge4$, so the projective dimension of $S/I_L$ is $3$. From the Auslander--Buchsbaum formula we have that the depth of $S/I_L$ is $1$. Thus, $S/I_L$ is a Cohen--Macaulay ring.
\end{remark}

\section{Set-theoretic complete intersection lattice ideals}
\label{arithmeticalrank}

Let $\alpha_{ij}$, $1\leq i\leq 3$ and $1\leq j\leq 4$, be positive
integers satisfying conditions~\textup{(A1)}, \textup{(A2)}, and \textup{(A3)} of
Definition~\ref{defadmissible}. Set
$\mathbf{c}_1=(\alpha_{11},-\alpha_{12},-\alpha_{13},-\alpha_{14})$,
$\mathbf{c}_2=(-\alpha_{21},-\alpha_{22},-\alpha_{23},\alpha_{24})$, and
$\mathbf{c}_3=(-\alpha_{31},\alpha_{32},-\alpha_{33},-\alpha_{34})$,
and let
$L=\mathbb{Z}\mathbf{c}_1+\mathbb{Z}\mathbf{c}_2+
\mathbb{Z}\mathbf{c}_3\subseteq\mathbb{Z}^4$.

\begin{notation} \label{basicnot} Throughout this section we use the following notation:
\begin{align*}
n_1&=\alpha_{12}\alpha_{23}\alpha_{34}
+\alpha_{12}\alpha_{24}\alpha_{33}
-\alpha_{13}\alpha_{22}\alpha_{34}
+\alpha_{13}\alpha_{24}\alpha_{32}
+\alpha_{14}\alpha_{22}\alpha_{33}
+\alpha_{14}\alpha_{23}\alpha_{32},\\
n_2&=\alpha_{11}\alpha_{23}\alpha_{34}
+\alpha_{11}\alpha_{24}\alpha_{33}
+\alpha_{13}\alpha_{21}\alpha_{34}
+\alpha_{13}\alpha_{24}\alpha_{31}
-\alpha_{14}\alpha_{21}\alpha_{33}
+\alpha_{14}\alpha_{23}\alpha_{31},\\
n_3&=\alpha_{11}\alpha_{24}\alpha_{32}
-\alpha_{11}\alpha_{22}\alpha_{34}
-\alpha_{12}\alpha_{21}\alpha_{34}
-\alpha_{12}\alpha_{24}\alpha_{31}
-\alpha_{14}\alpha_{21}\alpha_{32}
-\alpha_{14}\alpha_{22}\alpha_{31},\\
n_4&=\alpha_{11}\alpha_{22}\alpha_{33}
+\alpha_{11}\alpha_{23}\alpha_{32}
+\alpha_{12}\alpha_{21}\alpha_{33}
-\alpha_{12}\alpha_{23}\alpha_{31}
+\alpha_{13}\alpha_{21}\alpha_{32}
+\alpha_{13}\alpha_{22}\alpha_{31}.
\end{align*}
\end{notation}

\begin{lemma}\label{specificadmissible}
The vectors $\mathbf{c}_1$, $\mathbf{c}_2$ and $\mathbf{c}_3$ are
linearly independent and form an admissible basis of $L$.
Moreover, $n_i>0$ for $1\leq i\leq 4$.
\end{lemma}

\begin{proof}
Let $A$ be the $3\times4$ matrix whose rows are $\mathbf c_1$, $\mathbf c_2$, and $\mathbf c_3$. For $i=1,\ldots,4$, let $\Delta_i$ denote the
determinant of the submatrix of $A$ obtained by deleting the $i$-th column of $A$. Then
\[
\mathbf n=(n_1,n_2,n_3,n_4)
=(-\Delta_1,\Delta_2,-\Delta_3,\Delta_4).
\]
Hence $A\mathbf n^{\mathsf T}=\mathbf0$, where $\mathbf n^{\mathsf T}$ denotes the transpose of $\mathbf n$.

Since $\alpha_{24}>\alpha_{34}$ and $\alpha_{32}>\alpha_{22}$, we have
\[n_1=\alpha_{13}
 (\alpha_{24}\alpha_{32}-\alpha_{22}\alpha_{34})
 +\alpha_{12}\alpha_{23}\alpha_{34}
 +\alpha_{12}\alpha_{24}\alpha_{33}+\alpha_{14}\alpha_{22}\alpha_{33}
 +\alpha_{14}\alpha_{23}\alpha_{32}>0.\]
Using the inequalities $\alpha_{11}>\alpha_{21}$ and
$\alpha_{24}>\alpha_{14}$, we obtain
\[n_2=\alpha_{33}
 (\alpha_{11}\alpha_{24}-\alpha_{14}\alpha_{21})
 +\alpha_{11}\alpha_{23}\alpha_{34}
 +\alpha_{13}\alpha_{21}\alpha_{34}+\alpha_{13}\alpha_{24}\alpha_{31}
 +\alpha_{14}\alpha_{23}\alpha_{31}>0.\]
 Since $\alpha_{11}>\alpha_{31}$ and $\alpha_{32}>\alpha_{12}$, we have
\[n_4=\alpha_{23}
 (\alpha_{11}\alpha_{32}-\alpha_{12}\alpha_{31})
 +\alpha_{11}\alpha_{22}\alpha_{33}
 +\alpha_{12}\alpha_{21}\alpha_{33}+\alpha_{13}\alpha_{21}\alpha_{32}
 +\alpha_{13}\alpha_{22}\alpha_{31}>0.\]

Also
\begin{align*}
n_3={}&(\alpha_{11}-\alpha_{21}-2\alpha_{31})
 (\alpha_{24}\alpha_{32}-\alpha_{22}\alpha_{34})\\
&+\alpha_{21}\Bigl[
 (\alpha_{24}-\alpha_{14}-2\alpha_{34})\alpha_{32}
 +\alpha_{34}(2\alpha_{32}-\alpha_{12}-\alpha_{22})
 \Bigr]\\
&+\alpha_{31}\Bigl[
 (\alpha_{14}+2\alpha_{34})
 (2\alpha_{32}-\alpha_{12}-\alpha_{22})+(\alpha_{24}-\alpha_{14}-2\alpha_{34})
 (2\alpha_{32}-\alpha_{12})
 \Bigr]>0,
\end{align*}
since
$\alpha_{11}-\alpha_{21}-2\alpha_{31}>0$, $\alpha_{24}-\alpha_{14}-2\alpha_{34}>0$, $2\alpha_{32}-\alpha_{12}-\alpha_{22}>0$ and $\alpha_{24}\alpha_{32}-\alpha_{22}\alpha_{34}>0$.

From the fact that $n_1=-\Delta_1\neq0$, we obtain that the matrix $A$ has
rank three. As its rows generate $L$, they
form a $\mathbb{Z}$-basis of $L$.
\end{proof}

\begin{remark}\label{Posgrad}
We have $L\subseteq\ker_{\mathbb{Z}}(n_1,\ldots,n_4)$, so $I_L$ is homogeneous with respect to the grading $\deg(x_i)=n_i$ for $1 \le i \le 4$.
\end{remark}

Let $\mathbf{v}_1=\mathbf{u}_1$, $\mathbf{v}_2=\mathbf{u}_3$, $\mathbf{v}_3=-\mathbf{u}_7$, and $\mathbf{v}_4=\mathbf{u}_2$. The next lemma determines $\gamma_j(L)$ for $1 \le j \le 4$.

\begin{lemma}\label{gamma}
Let $\{\mathbf{c}_1,\mathbf{c}_2,\mathbf{c}_3\}$ be an admissible
basis of a lattice $L\subseteq\mathbb{Z}^4$. Then $\gamma_j(L)=(\mathbf{v}_j)_j$ for $1\leq j\leq4$.
In particular, $\gamma_1(L)=\alpha_{11}$, $\gamma_2(L)=\alpha_{32}$, $\gamma_3(L)=\alpha_{13}+\alpha_{23}+2\alpha_{33}$, and
$\gamma_4(L)=\alpha_{24}$.
\end{lemma}

\begin{proof}
Using conditions \textup{(A1)}, \textup{(A2)} and \textup{(A3)}, we obtain $\supp(\mathbf{v}_j^+)=\{j\}$. Suppose that the vector $\mathbf{w}\in L$ satisfies
$\supp(\mathbf{w}^+)=\{j\}$ and
$0<(\mathbf{w})_j\leq(\mathbf{v}_j)_j$.
Then $\mathbf{w}\leq\mathbf{v}_j^+$. By Theorem~\ref{computeneighbors}, $\mathbf{v}_1$,
$\mathbf{v}_2$,
$\mathbf{v}_3$ and
$\mathbf{v}_4$ are neighbors of the origin.
Thus,
$\mathbf{w}=\mathbf{0}$ or $\mathbf{w}=\mathbf{v}_j$. Since $(\mathbf{w})_j>0$, $\mathbf{w}$ cannot be equal to $\mathbf{0}$. Consequently, $\mathbf{w}=\mathbf{v}_j$, and therefore $\gamma_j(L)=(\mathbf{v}_j)_j$.
\end{proof}

\begin{theorem}\label{computeara}
Let $K$ be a field of any characteristic. Then $I_L$ is a
set-theoretic complete intersection.
\end{theorem}

\begin{proof}
In positive characteristic, $I_L$ is a set-theoretic complete intersection.

Suppose that the characteristic of the field $K$ is zero. By Proposition~\ref{positivelattice}, $L$ is a positive lattice. Moreover, from the proof of Lemma~\ref{gamma}, we have
$\supp(\mathbf{v}_j^+)=\{j\}$ for $1\leq j\leq4$. Furthermore, the set $\{\mathbf{v}_1,\mathbf{v}_2,\mathbf{v}_4\}$ generates $L$ and
\[
\mathbf{v}_1+2\mathbf{v}_2+\mathbf{v}_3+\mathbf{v}_4
=\mathbf{0}.
\]

By conditions \textup{(A1)}, \textup{(A2)}, and \textup{(A3)}, we have

\[
\begin{aligned}
\max_{j\ne1}|(\mathbf v_j)_1|
 &=\max\{\alpha_{31},\alpha_{21},
    \alpha_{11}-\alpha_{21}-2\alpha_{31}\}<\alpha_{11},\\
\max_{j\ne2}|(\mathbf v_j)_2|
 &=\max\{\alpha_{12},\alpha_{22},
    2\alpha_{32}-\alpha_{12}-\alpha_{22}\}<\alpha_{32},\\
\max_{j\ne3}|(\mathbf v_j)_3|
 &=\max\{\alpha_{13},\alpha_{23},\alpha_{33}\}
   <\alpha_{13}+\alpha_{23}+2\alpha_{33},\\
\max_{j\ne4}|(\mathbf v_j)_4|
 &=\max\{\alpha_{14},\alpha_{34},
    \alpha_{24}-\alpha_{14}-2\alpha_{34}\}<\alpha_{24}.
\end{aligned}
\]
Thus $|(\mathbf v_j)_i|<\gamma_i(L)$ for every $1 \leq i,j \leq 4$ with $i \neq j$. Hence Theorem~\ref{basicstci} applies, and therefore
$\ara(I_L)=\Ht(I_L)=3$.
\end{proof}

The next proposition characterizes when $I_L$ is a toric ideal.

\begin{proposition}\label{monomialcurve}
The following are equivalent:
\begin{enumerate}
\item[\textup{(i)}] $\gcd(n_1,n_2,n_3,n_4)=1$;
\item[\textup{(ii)}] $\mathbb{Z}^4/L$ is torsion free;
\item[\textup{(iii)}] $L=\ker_{\mathbb{Z}}(n_1,n_2,n_3,n_4)$.
\end{enumerate}
\end{proposition}

\begin{proof}
By Remark~\ref{Posgrad}, $L\subseteq\ker_{\mathbb{Z}}(n_1,n_2,n_3,n_4)$.
Since both lattices have rank three, the quotient
$\ker_{\mathbb{Z}}(n_1,n_2,n_3,n_4)/L$ is a finite subgroup of $\mathbb{Z}^4/L$. Also, $\mathbb{Z}^4/\ker_{\mathbb{Z}}(n_1,n_2,n_3,n_4)$ is torsion free. Thus $L=\ker_{\mathbb{Z}}(n_1,n_2,n_3,n_4)$ if and only 
if $\mathbb{Z}^4/L$ is torsion free.

Let $A$ be the $3\times4$
matrix whose rows are
$\mathbf{c}_1,\mathbf{c}_2,\mathbf{c}_3$. Since $A$ has rank 3, its Smith normal form is \[
\begin{pmatrix}
e_1&0&0&0\\
0&e_2&0&0\\
0&0&e_3&0
\end{pmatrix},
\]
where $e_1$, $e_2$, and $e_3$ are the nonzero invariant factors of $A$.  By the fundamental theorem of finitely generated
abelian groups
\cite[Section~12.1]{DuFoot},
\[
\mathbb{Z}^4/L\cong
\mathbb{Z}\oplus\mathbb{Z}/e_1\mathbb{Z}
\oplus\mathbb{Z}/e_2\mathbb{Z}
\oplus\mathbb{Z}/e_3\mathbb{Z}.
\]
The discussion following
Theorem~II.9 in~\cite{Nebook} shows that
$e_1e_2e_3=\gcd(n_1,n_2,n_3,n_4)$.
Consequently, $\gcd(n_1,n_2,n_3,n_4)=1$ if and only if $e_1=e_2=e_3=1$, namely $\mathbb{Z}^4/L$ is torsion free.
\end{proof}

\begin{example}
Let $\mathbf{c}_1=(5,-2,-1,-1)$, $\mathbf{c}_2=(-1,-2,-2,4)$, and $\mathbf{c}_3=(-1,3,-1,-1)$. Here $n_1=30$, $n_2=36$, $n_3=35$ and $n_4=43$, and $\gcd(n_1,\ldots,n_4)=1$. By Proposition~\ref{monomialcurve} $I_L$ is a toric ideal. In particular, $I_L$ is the defining ideal of the monomial curve parametrized by
$x_1=t^{30}$, $x_2=t^{36}$, $x_3=t^{35}$, and $x_4=t^{43}$. The ideal $I_L$ is minimally generated by $x_1^5-x_2^2x_3x_4$, $x_4^4-x_1x_2^2x_3^2$, $x_2^3-x_1x_3x_4$, $x_1^4x_2-x_3^2x_4^2$, $x_2x_4^3-x_1^2x_3^3$, $x_1^3x_4^2-x_2x_3^4$, and $x_1^2x_2^2x_4-x_3^5$. By Theorem~\ref{computeara}, $I_L$ is a
set-theoretic complete intersection.
\end{example}

Finally, we construct an infinite family of rank-three lattices in $\mathbb Z^4$ with admissible bases. For an integer $m\geq 1$, let
\begin{align*}
\mathbf c_1(m)&=(m+4,-(m+1),-1,-1),\\
\mathbf c_2(m)&=(-(m+1),-3,-1,m+4),\\
\mathbf c_3(m)&=(-1,m+3,-(m+1),-1),
\end{align*}
and set
\[
L(m)=\mathbb Z\mathbf c_1(m)+\mathbb Z\mathbf c_2(m)
     +\mathbb Z\mathbf c_3(m).
\]

\begin{remark} \label{Rembas} Let $\alpha_{11}=\alpha_{24}=m+4$, $\alpha_{12}=\alpha_{21}=\alpha_{33}=m+1$, $\alpha_{22}=3$, $\alpha_{32}=m+3$, $\alpha_{13}=\alpha_{14}=\alpha_{23}=\alpha_{31}=\alpha_{34}=1$. Every $\alpha_{ij}$ is positive. Also $\alpha_{11}-\alpha_{21}-2\alpha_{31}=1>0$, so the condition \textup{(A1)} holds. The inequalities
\[
\max\{\alpha_{12},\alpha_{22}\}=\max\{m+1,3\}
 <m+3=\alpha_{32}<m+4=\alpha_{12}+\alpha_{22}
\]
verify \textup{(A2)}, while
\[
\alpha_{24}-\alpha_{14}-2\alpha_{34}=m+1>0
\]
verifies \textup{(A3)}. Thus $\{\mathbf c_1(m),\mathbf c_2(m),\mathbf c_3(m)\}$ is an admissible basis for the lattice $L(m)$.
\end{remark}

The integers $n_1(m),\ldots,n_4(m)$ obtained from Notation~\ref{basicnot} are
$n_1(m)=m^3+7m^2+21m+20$, $n_2(m)=m^3+8m^2+25m+25$, $n_3(m)=m^3+8m^2+26m+25$, and $n_4(m)=m^3+8m^2+28m+30$.

\begin{proposition}\label{basicsfamily}
For every integer $m\geq 1$, the following statements hold.
\begin{enumerate}
\item[\textup{(i)}]
\[
\gcd\bigl(n_1(m),n_2(m),n_3(m),n_4(m)\bigr)
=
\begin{cases}
5, & \text{if } 5\ \textrm{divides} \ m,\\
1, & \text{if } 5\ \textrm{does not divide} \ m.
\end{cases}
\]

\item[\textup{(ii)}]
If $5$ does not divide $m$, then
\[
L(m)=\ker_{\mathbb Z}(n_1(m),n_2(m),n_3(m),n_4(m)).
\]
\end{enumerate}
\end{proposition}

\begin{proof} \textup{(i)} Let $d=\gcd(n_1(m),n_2(m),n_3(m),n_4(m)$. Since
$n_3(m)-n_2(m)=m$ and $n_4(m)-n_3(m)=2m+5$, $d$ divides both $m$ and $2m+5$, and therefore $d$ divides also $5$. Thus $d=1$ or $d=5$.
If $5$ divides $m$, then $5$ divides $n_i$ for every $1 \leq i \leq 4$. Thus $d=5$. Suppose that $5$ does not divide $m$. Since $d$ divides $m$, we get $d=1$.

\textup{(ii)} If $5$ does not divide $m$, then \textup{(i)} gives
$\gcd(n_1(m),n_2(m),n_3(m),n_4(m))=1.$ Proposition~\ref{monomialcurve} therefore yields
\[
L(m)=\ker_{\mathbb Z}
\bigl(n_1(m),n_2(m),n_3(m),n_4(m)\bigr). \qedhere
\]
\end{proof}

Consider the graded reverse lexicographic order $\prec$ on $S$
with $\deg(x_i)=n_i(m)$ and $x_1\succ x_2\succ x_4\succ x_3$. Then 
$B(\mathbf u_1)=x_1^{m+4}-x_2^{m+1}x_3x_4$, $B(\mathbf u_2)=x_4^{m+4}-x_1^{m+1}x_2^3x_3$, $B(\mathbf u_3)=x_2^{m+3}-x_1x_3^{m+1}x_4$, $B(\mathbf u_4)=x_1^{m+3}x_2^2-x_3^{m+2}x_4^2$, $B(\mathbf u_5)=x_2^m x_4^{m+3}-x_1^{m+2}x_3^{m+2}$, $B(\mathbf u_6)=x_1^2x_4^{m+2}-x_2x_3^{m+3}$, and
$B(\mathbf u_7)=x_1x_2^{m+2}x_4^{m+1}-x_3^{2m+4}$.

\begin{proposition}\label{family-properties}
For every integer $m\geq1$, the set $G=\{B(\mathbf u_1),\ldots,B(\mathbf u_7)\}$
is a Gr\"obner basis of $I_{L(m)}$ with respect to $\prec$.
Moreover, $I_{L(m)}$ is generic,
\[
N(L(m))=\{\pm\mathbf u_1,\ldots,\pm\mathbf u_7\},
\]
and each $B(\mathbf u_i)$, $1 \leq i \leq 7$, is an indispensable binomial of $I_{L(m)}$.
The Betti numbers of $S/I_{L(m)}$ are $\beta_0(S/I_{L(m)})=1$, $\beta_1(S/I_{L(m)})=7$, $\beta_2(S/I_{L(m)})=12$, and $\beta_3(S/I_{L(m)})=6$
with $\beta_i(S/I_{L(m)})=0$ for $i\geq4$. In addition,
$I_{L(m)}$ is a set-theoretic complete intersection.
\end{proposition}

\begin{proof}
By Remark~\ref{Rembas},
$\{\mathbf c_1(m),\mathbf c_2(m),\mathbf c_3(m)\}$ is an admissible
basis of $L(m)$. The fact that $G$ is a Gr\"obner basis of $I_{L(m)}$ with respect to $\prec$ and that $I_{L(m)}$ is generic follows
from Theorem~\ref{genericlattice}. The description of $N(L(m))$ and
the fact that $B(\mathbf u_1),\ldots,B(\mathbf u_7)$ are indispensable binomials follow from
Theorem~\ref{computeneighbors}. Corollary~\ref{computebetti} provides
the Betti numbers of $S/I_{L(m)}$, while Theorem~\ref{computeara} guarantees the
set-theoretic complete-intersection property of $I_{L(m)}$.
\end{proof}

\begin{remark}\label{Ojeda-Pisonex}
For $m=1$, the four integers are $n_1(1)=49$, $n_2(1)=59$,
$n_3(1)=60$, and $n_4(1)=67$. Thus the corresponding affine monomial curve is parametrized by $x_1=t^{49}$, $x_2=t^{59}$, $x_3=t^{60}$, and $x_4=t^{67}$. This is the curve considered in Example~4.3 of \cite{OjPi}. As observed there, the lattice vectors corresponding to the minimal binomial generators of $I_{L(1)}$ do not form the vertices of a three-dimensional unit cube.
\end{remark}

\end{document}